\documentclass[reqno,12pt,hidelinks]{amsart}
\usepackage{amssymb,amsmath,amsthm,amsxtra,calc,bm, color, verbatim}
\usepackage{enumerate}
\usepackage[margin=.95in]{geometry}

\usepackage[scr=boondoxo]{mathalfa}
\usepackage{mathtools}

\usepackage{mleftright}
\mleftright

\usepackage{nccmath}
\usepackage{cases}

\usepackage{calc}
\usepackage{graphicx}

\usepackage{hyperref}

\usepackage[final]{microtype}

\renewcommand{\(}{\left(}
\renewcommand{\)}{\right)}
\newcommand{\spmod}[1]{\ensuremath{\,(#1)}}

\renewcommand{\|}{\big |}

\def\Z{\mathbb{Z}}
\def\Q{\mathbb{Q}}

\def\H{\mathbb{H}}

\def\C{\mathbb{C}}

\def\F{\mathbb{F}}

\def\SL{{\rm SL}}

\def\GL{{\rm GL}}

\newcommand{\pfrac}[2]{\left(\frac{#1}{#2}\right)}
\newcommand{\ptfrac}[2]{\left(\tfrac{#1}{#2}\right)}
\newcommand{\pMatrix}[4]{\left(\begin{matrix}#1 & #2 \\ #3 & #4\end{matrix}\right)}

\renewcommand{\pmatrix}[4]{\left(\begin{smallmatrix}#1 & #2 \\ #3 & #4\end{smallmatrix}\right)}
\renewcommand{\bar}[1]{\overline{#1}}

\renewcommand{\sl}{\big| }
\DeclareMathOperator{\new}{new}

\def\ep{\varepsilon}

\newtheorem{theorem}{Theorem}[section]
\newtheorem{lemma}[theorem]{Lemma}

\newtheorem{proposition}[theorem]{Proposition}

\theoremstyle{remark}
\newtheorem*{remark}{Remark}

\newtheorem*{example}{Example}

\newtheorem*{definition}{Definition}
\numberwithin{equation}{section}

\mathtoolsset{showonlyrefs}

\newcommand{\Gal}{\operatorname{Gal}}

\def\cpm{c\phi_m}
\def\cp13{c\phi_{13}}
\def\bl{\bar\ell}

\title{Congruences like Atkin's for generalized Frobenius partitions}

\date{\today}
\date{\today}
\author{Scott Ahlgren}
\address{Department of Mathematics\\
University of Illinois\\
Urbana, IL 61801} 
\email{sahlgren@illinois.edu} 

 \author{Nickolas Andersen}
 \address{Department of Mathematics \\ 
 Brigham Young University, Provo, UT 84602}
 \email{nick@math.byu.edu}
\author{Robert Dicks}
\address{Department of Mathematics\\
Clemson University\\
Clemson, SC 29634} 
\email{rdicks@clemson.edu}  
 
\thanks{The first author was  supported by a grant from the Simons Foundation (\#963004 to Scott Ahlgren).
The second author was  supported by a grant from the Simons Foundation (\#854098 to Nickolas Andersen)}
\begin{document}

\begin{abstract}

In the 1960s Atkin discovered congruences modulo primes $\ell
\leq 31$ for the partition function $p(n)$ in arithmetic progressions modulo $\ell Q^3$, where $Q\neq \ell$ is prime.
Recent work of the first author with Allen and Tang shows that such congruences exist for all primes $\ell\geq 5$.

Here we consider (for primes $m\geq 5$) the 
$m$-colored generalized Frobenius partition functions $c\phi_m(n)$; these are  natural level $m$ analogues of $p(n)$.  
For each such $m$ 
we prove that there are similar congruences for $\cpm(n)\pmod \ell$ for all primes $\ell$ outside of an explicit finite set depending on $m$.
To prove the result we first construct, using  both theoretical and computational methods, cusp forms of half-integral weight on $\Gamma_0(m)$ which capture the relevant values of $c\phi_m(n)$ modulo~$\ell$.
We then apply previous work of the authors on the Shimura lift for modular forms with the eta multiplier
together with tools from 
the theory of modular Galois representations.
\end{abstract}

\maketitle

\setcounter{tocdepth}{2}

\section{Introduction and statement of results}

Since their discovery over a century ago, the Ramanujan congruences
\begin{equation} \label{eq:ram-cong}
    p\pfrac{\ell n+1}{24} \equiv 0 \pmod{\ell}, \qquad \ell=5,7,11,
\end{equation}
for the partition function $p(n)$ have been the catalyst for a rich theory of congruences for combinatorial sequences coming from modular forms or related objects.
(We agree that $p(n) = 0$ if $n\notin \Z_{\geq 0}$.)
Ramanujan, Watson, and Atkin \cite{Atkin1, Ramanujan, Ramanujan2, Watson} extended \eqref{eq:ram-cong} to arbitrary powers of $5$, $7$, and $11$.

For primes $\ell \leq 31$, Atkin \cite[eq.~$(52)$]{Atkin2} found examples of congruences of the form
\begin{equation} \label{eq:atkin-cong}
    p\pfrac{\ell Q^2 n+\beta}{24} \equiv 0 \pmod{\ell} \quad \text{ if } \ \ \pfrac nQ = \ep_Q,
\end{equation}
where $Q$ is a prime different from $\ell$ and $\ep_Q \in \{\pm 1\}$.
This leads to many congruences of the form $p(\ell Q^3 n+\alpha)\equiv 0\pmod\ell$ for $\ell\leq 31$.
Ono  \cite{Ono}  proved the existence of infinitely many congruences of the form
 $p(\ell Q^4 n+\alpha) \equiv 0 \pmod\ell$
 for any prime $\ell\geq 5$ (Treneer \cite{Treneer} has shown that such congruences exist for a very general class of weakly holomorphic modular forms). 
On the other hand, 
the results of \cite{ahlgren-beckwith-raum} show that for a given $\ell$ there are very few, if any,
congruences of the form $p(\ell Q n+\alpha)\equiv 0\pmod\ell$
or 
$p(\ell Q^2 n+\alpha)\equiv 0\pmod\ell$.

Until  recently, it was not known whether there were congruences like Atkin's \eqref{eq:atkin-cong} for primes larger than $31$.
In joint work with Allen and Tang \cite{Ahlgren-Allen-Tang}, the first author proved that for every prime $\ell\geq 5$, there exists a set of primes $Q$ of positive density for which congruences of the form \eqref{eq:atkin-cong} hold.

The proofs of the  results described above  rely on the fact that the generating function for $p(n)$ is a modular form of weight $-1/2$ on the  modular group $\SL_2(\Z)$.
In this paper we will study the $m$-colored generalized Frobenius partition function $\cpm(n)$ where $m$ is a positive integer.   This is a natural higher-level analogue of $p(n)$; we have $c\phi_1(n)=p(n)$ \cite[$(5.15)$]{Andrews}, and the 
 generating function for $\cpm(n)$ is a modular form of weight $-1/2$ on $\Gamma_0(m)$ (if $m$ is odd,  which is the only relevant case for us).
 
To define these partition functions requires  some notation.
To each partition of $n$, we can associate an array
\begin{equation} \label{eq:frob-symb}
    \left( \begin{array}{cccc}
    a_1 & a_2 & \cdots & a_r \\
    b_1 & b_2 & \cdots&b_r\end{array} 
    \right)
\end{equation}
of non-negative integers with strictly decreasing rows and with
\begin{equation} \label{eq:n-sum-frob-symb}
    n=r+\sum^{r}_{i=1}a_i+\sum^{r}_{i=1} b_i.
\end{equation}
The number $r$ is the length of the diagonal in the Ferrers diagram for the partition, the $a_i$ are the numbers of dots to the right of the diagonal in each row, and the $b_i$ are the numbers of dots below the diagonal in each column.
The array \eqref{eq:frob-symb} is called a Frobenius symbol.

In \cite{Andrews}, Andrews introduced $m$-colored generalized Frobenius symbols.
These are symbols \eqref{eq:frob-symb} in which 
entries are colored with one of $m$ colors,  these colors are given an ordering, and  the rows retain the property of being strictly decreasing.
Let $\cpm(n)$ denote the number of $m$-colored generalized Frobenius symbols which sum to $n$ via \eqref{eq:n-sum-frob-symb}.
For example, the nine $2$-colored generalized Frobenius symbols of $2$ are
\begin{gather}
    \binom{1_1}{0_1},\ \binom{1_1}{0_2},\ \binom{1_2}{0_1},\ \binom{1_2}{0_2}, \binom{0_1}{1_1},\ \binom{0_1}{1_2},\ \binom{0_2}{1_1},\ \binom{0_2}{1_2},\ \pMatrix{0_2}{0_1}{0_2}{0_1},
\end{gather}
where we have imposed the ordering in which the color ``2'' is greater than the color ``1''.

The generating function for $c\phi_m(n)$ is (see \cite[Theorem 5.2]{Andrews})
\begin{equation}\label{eq:cphidef}
\sum c\phi_m\pfrac{n+m}{24}q^\frac n{24}=\eta^{-m}(z)\sum_{n=0}^\infty r_m(n)q^n,
\end{equation}
where $r_m(n)$ is the number of representations of $n$ by the quadratic form
\begin{equation}\label{eq:r_Qdef}
    \sum_{i=1}^{m-1}x_i^2+\sum_{1\leq i<j\leq m-1}x_ix_j
\end{equation}
and $\eta(z)$ is defined in \eqref{eq:etadef} below
(note that when $m=1$ the infinite sum in \eqref{eq:cphidef} becomes the number $1$).

In \cite{Chan-Wang-Yang}, Chan, Wang and Yang developed many modular properties of the generating functions \eqref{eq:cphidef} and made an intensive study of the generating functions for $m\leq 17$ 
(the bibliography of \cite{Chan-Wang-Yang} also contains many references to previous work on this topic).
For small primes $m$, the generating function for $c\phi_m(n)$ is particularly simple.
Equations (1.13) -- (1.15) of \cite{Chan-Wang-Yang} state that for $m\in \{5,7,11\}$ we have
\begin{equation}\label{eq:cpmram}
    c\phi_m(n) = p(n/m) + m p(mn-\beta_m), \qquad \text{where } \beta_m = \tfrac{m^2-1}{24}.
\end{equation}
From these and \eqref{eq:ram-cong} we easily obtain  six analogues of the Ramanujan congruences \eqref{eq:ram-cong}:
\begin{equation} \label{eq:cpramcong}
    \begin{aligned}
    c\phi_5\pfrac{\ell n+5}{24} &\equiv 0 \pmod{\ell}, \qquad \ell=7,11,\\
    c\phi_7\pfrac{\ell n+7}{24} &\equiv 0 \pmod{\ell}, \qquad \ell=5, 11,\\
    c\phi_{11}\pfrac{\ell n+11}{24} &\equiv 0 \pmod{\ell}, \qquad \ell=5, 7.
    \end{aligned}
\end{equation}

The theta function  $\sum r_m(n)q^n$ is a holomorphic modular form of weight $(m-1)/2$ and level $m$ or $2m$.  It 
follows from the results of 
Treneer \cite{Treneer} that if $\ell\geq 5$ is a prime with $\ell\nmid m$ and $j$ is a positive integer,  then there are infinitely many $Q$ giving rise to  congruences of the form
\begin{equation}\label{eq:treneerfrob}
    c\phi_m\pfrac{\ell^kQ^3 n+m}{24}\equiv 0\pmod {\ell^j} \quad \text{ if } (n, \ell Q)=1
\end{equation}
where $k$ is sufficiently large (see \cite[Theorem~2.1]{Chan-Wang-Yang}, \cite[Lemma~1, Corollary~1]{jameson_wieczorek} for details).

In \cite{AAD}, the authors developed a precise Shimura lift for modular forms which transform like powers of the eta function and developed some  arithmetic applications which generalize the results of \cite{Ahlgren-Allen-Tang} to a wide class of half-integral weight modular forms.
In this paper we build on those  results to
show, for each prime $m\geq 5$, that $\cpm(n)$ has congruences like Atkin's \eqref{eq:atkin-cong} for all primes $\ell$ outside of an explicit finite set.
In view of  the linear congruences \eqref{eq:cpramcong}, this is  interesting only if one of $\ell,m$ is at least $13$. We  note \cite[Theorem 10.2]{Andrews} that the situation is  simpler when $\ell=m$.  We also note that an extended example in the case $m=5$, $\ell=13$ is given in \cite[\S8]{AAD}.

Throughout the paper we will assume that $\ell, m\geq 5$ are  distinct primes and we define $\bl\in \{1, \dots, 23\}$ by $\bl\equiv \ell\pmod{24}$.  
For general $m$ we will operate under the hypothesis
\begin{equation}\label{eq:lmcond}
 \ell\bar{\ell} > m^2.
\end{equation}
For small values of $m$ it is possible to relax this hypothesis.  For example
\begin{equation}\label{eq:smalllmcond}
\text{If $m=13$ we require $\ell\geq 7$.}\ \ \ \text{If  $m=5, 7, 11$ we require only $\ell\neq m$}.
\end{equation}
Note that \eqref{eq:smalllmcond} includes such pairs as $m=11$, $\ell=97$ or $m=13$, $\ell=29$ which do not satisfy \eqref{eq:lmcond}.
With additional work one could 
relax the hypothesis \eqref{eq:lmcond} for other small values of $m$.  We now state the main results.

 \begin{theorem}\label{thm:frob1}
 Suppose that $\ell, m \geq 5$ are  distinct primes satisfying 
 \eqref{eq:lmcond} or \eqref{eq:smalllmcond}.
 Then there exists a positive density set $S$ of primes such that if $p \in S$ then $p \equiv 1 \pmod\ell $ and
\begin{equation}
 c\phi_m\(\frac{\ell p^2n+m}{24}\) \equiv 0 \pmod\ell  \ \ \ \text{ if }
\ \ \ \pfrac{n}{p}=
\pfrac{-1}{p}^{\frac{m\ell+1}{2}}\pfrac{p}{m}.
\end{equation}
 \end{theorem}

\begin{theorem}\label{thm:frob2}
  Suppose that $\ell, m \geq 5$ are  distinct primes satisfying 
 \eqref{eq:lmcond} or \eqref{eq:smalllmcond}. 
 Suppose further that there exists $a\in \Z$ such that $2^a\equiv -2\pmod\ell$.
 Then there exists a positive density set $S$ of primes and $\ep_{p} \in \{\pm 1\}$ such that if $p \in S$ then $p \equiv -2 \pmod{\ell}$ and 
 \begin{equation}
 c\phi_{m}\(\frac{\ell p^2n+m}{24}\) \equiv 0 \pmod{\ell}  \ \ \text{ if } \ \ \ \(\frac{n}{p}\)=\ep_{p}.
 \end{equation}
 \end{theorem}
We note that the sets $S$ in these results are  frobenian  in the sense of Serre \cite[\S 3.3]{serre-NXp}  (we do not prove that the set of all primes satisfying the conclusions is frobenian).  See the end of Section~1 of \cite{Ahlgren-Allen-Tang} for a more complete discussion.
We also note that by a result of Hasse \cite{Hasse}, the proportion of primes $\ell$ for which there exists $a$ with 
$2^a\equiv -2\pmod\ell$ is $17/24$.

Let $S_k(N, \omega)$  denote the space of cusp forms of weight $k$ and multiplier system $\omega$ on $\Gamma_0(N)$, and let $\nu$ denote the multiplier system for the Dedekind eta function (see 
\eqref{eq:etamult} below).
An important tool to study congruences for $p(n)=c\phi_1(n)$ is the fact that for any $\ell\geq 5$ there
is a modular form $f_\ell\in S_{(\ell-2)/2}\left(1, \nu^{-\ell}\right)$ 
such that 
\begin{gather} \label{eq:level1-fell}
    f_\ell \equiv \sum p\pfrac{\ell n+1}{24}q^\frac n{24} \pmod \ell.
\end{gather}
The existence of an $f_\ell$ satisfying \eqref{eq:level1-fell} follows from 
a  construction using the theory of modular forms modulo $\ell$ (see \cite{ahlgren-beckwith-raum}; in particular the remark following (1.8) and the discussion in Section~2).
Our theorems rely on the analogous result for $\cpm$ for primes $m\geq 5$.  The construction of these forms (which is given in Section~\ref{sec:construct}) is by contrast much more involved, and requires a mixture of theoretical and computational methods.

\begin{theorem} \label{thm:Flconstruct}   Suppose that $\ell, m \geq 5$ are  distinct primes satisfying 
 \eqref{eq:lmcond} or \eqref{eq:smalllmcond}.
Then there is a modular form  $F_\ell\in S_{(\ell-2)/2}\(m,\(\frac{\bullet}{m}\)\nu^{-m\ell}\)$ such that
\begin{equation} F_\ell(z)\equiv \sum \cpm\pfrac{\ell n+m}{24}q^\frac n{24}\pmod\ell.
\end{equation}
\end{theorem}

With  additional work, some of the results above can likely be generalized  to produce congruences for $\cpm(n)$ with powers of  $\ell$, square-free values of $m$, and  $n$ in other congruence classes modulo $\ell$.

\section{Background}\label{sec:GFP}
Here we briefly recall some essential facts (a more complete treatment of some of these facts can be found in 
\cite[\S3]{AAD}).
If $f$ is a function on the upper half-plane $\H$,  $k\in \frac12\Z$,  and $\gamma=\pmatrix abcd\in \GL_2^+(\Q)$, 
we define the slash operator in weight $k$ by 
\begin{equation}
\(f\|_k\gamma\)(z)=(\det\gamma)^\frac k2(cz+d)^{-k}f(\gamma z),
\end{equation}
where we always take the principal branch of the square root.
Let  $N$ be a positive integer and let $\omega$ be a multiplier on $\Gamma_0(N)$.  If $A$ is a subring of $\C$ then we denote by 
$M_k(N, \omega, A)$ the $A$-module of functions on $\H$  which  satisfy the transformation law
\begin{equation}
f\|_k\gamma=\omega(\gamma)f \quad \text{ for all } \gamma \in \Gamma_0(N),
\end{equation}
 which are holomorphic on $\H$ and at the cusps,
and whose Fourier coefficients lie in $A$.
We denote the subspace of cusp forms by   $S_k(N, \omega, A)$, 
and the space of weakly holomorphic  modular  forms (which are allowed poles at the cusps) by $M_k^!(N, \omega, A)$.
Throughout the paper $\ell\geq 5$ will be a fixed prime.  
When $A\subset \C$ is the ring of algebraic integers we will omit it from the notation; we will also omit the multipler $\omega$ from the notation when it is trivial.

We will be primarily concerned with the eta-multiplier $\nu$, which is defined
by 
\begin{equation}\label{eq:etamult}
\eta\|_\frac12\gamma=\nu(\gamma)\eta, \qquad \gamma\in \SL_2(\Z),
\end{equation}
where the Dedekind eta function is the modular form of weight $1/2$ on $\SL_2(\Z)$ defined by 
 \begin{equation}\label{eq:etadef}
  \eta(z)=q^\frac1{24}\prod_{n=1}^\infty(1-q^n), \qquad q:=e^{2\pi i z}.
  \end{equation}

We will  require Hecke theory only for the spaces $M_k(N)$ (where $k$ is an even integer).  
If $f=\sum a(n)q^n\in M_k(N)$ and $p\nmid N$
then we have the Hecke operator
$T_p: M_k(N)\to M_k(N)$ defined by 
\begin{equation}
    f\sl T_p=\sum\(a(pn)+p^{k-1}a\pfrac np\)q^n.
\end{equation}
For any prime $p$ we have the $U$ and $V$ operators defined by 
\begin{equation}
    f\sl U_p=\sum a(pn)q^n\ \ \ \ f\sl V_p=\sum a(n)q^{pn}.
\end{equation}
If $p\mid N$ then $f\sl U_p\in M_k(N)$, and for any $p$ we have $f\sl U_p, f\sl V_p\in M_k(pN)$.

For primes $p \mid N$, we define the space of forms $S^{\new p}_k(N,\C)$ which are $p$-new  by 
\begin{equation}
S^{\new p}_k(N,\C):= \bigoplus_{td \mid N, \ p \mid t}S^{\new}_{k}(t,\C) \sl V_d.
\end{equation}
By $S_k^{\new}(N)$, $S^{\new p}_k(N)$, etc. we mean the corresponding subspaces of forms with algebraic integer coefficients.

 If $p\mid N$ and $(p, N/p)=1$, then    an Atkin-Lehner matrix $W^N_p$ is any integral matrix with 
\begin{equation}\label{eq:atkinlehner}
W^N_p =\pMatrix{pa }b{Nc}{p }, \ \ \ \ a,b,c \in \Z, \ \ \ \ \det(W^N_p)=p.
\end{equation}
Suppose that   $f\in S^{\new p}_k(N)$ is a newform; i.e. there exists a level $N'$ with $p\mid N'\mid N$
such that 
$f\in S^{\new}_k(N')$ and $f$ is  a normalized eigenform of the operators $T_Q$ for $Q\nmid N'$ and $U_Q$ for $Q\mid N'$.   Then there is an Atkin-Lehner eigenvalue $\ep_p\in \{\pm 1\}$ such that
$f\sl_k W_p^N=f\sl_k W_p^{N'}=\ep_p f$; if $f=\sum a(n)q^n$ then $\ep_p=-p^{1-k/2}a(p)$ \cite[Theorem $3$]{Atkin-Lehner}.

We will also  require some standard facts about filtrations.
 Suppose that $\ell \geq 5$ is prime. 
 If $f =\sum a(n)q^n\in \Z[\![q]\!]$  then we define
\begin{equation}
\bar f:= \sum \bar{a(n)}q^n \in \mathbb F_\ell [\![q]\!].
\end{equation}
If  $\bar f$ is the reduction modulo $\ell$ of an element of $M_k(N, \Z)$ for some $k$ then we define the filtration
\begin{equation}
w_\ell (\bar f):= \inf\{ k':  \ \text{there exists $g \in M_{k'}\(N,\Z\)$ with $\bar f=\bar  g$}\}.
\end{equation} 
Then we have  the following  (see e.g. \cite[\S 1]{Jochnowitz}).
\begin{lemma}\label{lem:filt}
  Suppose that $\ell\geq 5$ is prime and that $f \in M_ k\(N,\Z\)$. Then we have  
\begin{enumerate}
\item
If $g \in M_{k'}\(N,\Z\)$ has $\bar f=\bar  g$ then  $k \equiv k'\pmod{\ell-1}$.
\item $w_\ell\(\bar{f\sl U_{\ell}} \)\leq (w_\ell(\bar f)-1)/\ell+\ell$.
\end{enumerate}
\end{lemma}

\section{Proof of Theorem~\ref{thm:Flconstruct}}
\label{sec:construct}

Using the connection to quadratic forms described in \eqref{eq:cphidef},
Chan, Wang and Yang \cite[Thm.~$2.1$]{Chan-Wang-Yang} showed that if $m$ is a positive odd integer, then 
\begin{equation}\label{eq:Amdef}
 A_m(z):=\prod_{n \geq 1}(1-q^n)^m\sum^\infty_{n=0} \cpm(n)q^n \in M_\frac{m-1}2\(m, \ptfrac\bullet m\).
\end{equation}

We  first address the main case when the inequality $\ell\bl>m^2$ holds.
We begin with a result which has a straightforward generalization to square-free $m$ which are coprime to $6$; for simplicity we state it only in the case when $m$ is prime.
\begin{lemma}\label{lem:flconstruct}
Let  $\ell, m\geq 5$ be  primes satisfying
\eqref{eq:lmcond}.
Then there exists $f_\ell\in S_{(\ell+\bl-2)/2}(m, \Z)$ such that
\begin{equation} f_\ell(z)\equiv \eta^{\bl}(mz)\sum \cpm\pfrac{\ell n+m}{24}q^\frac n{24}\pmod\ell.
\end{equation}

\end{lemma}

\begin{proof} Define 
\begin{equation} \label{eq:hldef} h_\ell(z):=\frac{\eta^{\ell\bl}(mz)}{\eta^m(z)}A_m(z)
=\eta^{\ell\bl}(mz)\sum \cpm\pfrac{ n+m}{24}q^\frac n{24}
=q^\frac{(\ell\bl-1)m}{24}+\cdots .
\end{equation}
Using \eqref{eq:Amdef} and  a standard 
 criterion for eta-quotients (see e.g. \cite[Thm. 1.64]{ono_web}) we see that $h_\ell\in M^!_{(\ell\bl-1)/2}\(m,\Z\)$.  
 It is clear that $h_\ell$ vanishes at $\infty$; to check that it vanishes at the 
 other cusp of $\Gamma_0(m)$ we use the fact that $\eta(-1/z)=\sqrt{z/i}\,\eta(z)$ to see that 
 \begin{gather}
\frac{\eta^{\ell\bl}(mz)}{\eta^m(z)}\Big|_\frac{\ell\bl-m}2\pMatrix{0}{-1}{m}{0}
 =cq^\frac{\ell\bl-m^2}{24}+\cdots\ \ \ \ \text{with $c\neq 0$.}
 \end{gather}
 From \eqref{eq:lmcond}   it follows   that  $h_\ell\in S_{(\ell\bl-1)/2}\(m,\Z\)$.

By \eqref{eq:hldef} we have 
\begin{equation}
h_\ell\sl U_\ell\equiv \eta^{\bl}(mz)\sum \cpm\pfrac{\ell n+m}{24}q^\frac n{24}\pmod\ell.
\end{equation}
 It follows from Lemma~\ref{lem:filt} and the congruence $h_\ell\sl U_\ell\equiv h_\ell\sl T_\ell\pmod \ell$ that 
\begin{equation}
w\(\bar{h_\ell \sl U_\ell}\)\equiv \frac{\ell\bl-1}2\pmod{\ell-1}
\end{equation}
and that
\begin{equation}
w\(\bar{h_\ell \sl U_\ell}\)\leq \frac1\ell\(\frac{\ell\bl-3}2\)+\ell=\ell+\frac\bl2-\frac3{2\ell}.
\end{equation}
Note that  $(\ell+\bl-2)/2$ is the unique positive integer satisfying both of these conditions.  It follows that
there exists $f_\ell\in S_{(\ell+\bl-2)/2}\(m,\Z\)$ such that $f_\ell\equiv h_\ell\sl U_\ell\pmod\ell$, which   proves the lemma.
\end{proof}
To deduce Theorem~\ref{thm:Flconstruct} in this case it would suffice to show that the quotient $f_\ell/\eta^{\bl}(mz)$ is a cusp form.  Here there is a significant obstruction.  Since $\ell>m$   it is clear from Lemma~\ref{lem:flconstruct} that 
$f_\ell$ vanishes modulo $\ell$ to order $\geq\lceil\bl m/24\rceil$ at $\infty$, but we have no guarantee that $f_\ell$ itself vanishes to this order (there are many modular forms whose order of vanishing modulo $\ell$ is strictly greater than their order of vanishing).
To overcome this obstruction we use the following result.
\begin{proposition} \label{prop:ech} Let  $\ell, m\geq 5$ be  primes satisfying \eqref{eq:lmcond}. 
  Then there exist cusp forms $G_1, G_2, \dots,  G_{\left\lfloor\bl m/24\right\rfloor}\in S_{(\ell+\bl-2)/2}(m)$ with $\ell$-integral rational coefficients such that  
\begin{equation}\label{eq:echbasis}
G_j=q^j +O(q^{j+1}), \ \ \ \ j=1, \dots, \left\lfloor\frac{\bl m}{24}\right\rfloor.
\end{equation}
\end{proposition}
Before proving the proposition we show that Theorem~\ref{thm:Flconstruct}
is a consequence in the case that \eqref{eq:lmcond} holds.
\begin{proof}[Proof of Theorem~\ref{thm:Flconstruct} when $\ell\bl>m^2$]

Let $f_\ell\in S_{(\ell+\bl-2)/2}(m, \Z)$ be the form from Lemma~\ref{lem:flconstruct}; by the discussion above   
 we have 
\begin{equation} f_\ell=a_1q+a_2q^2+\dots, \ \ \ \ \text{where $a_j\equiv 0\pmod\ell$ \ \  \ if \ \ \ $j\leq \left\lfloor\frac{\bl m}{24}\right\rfloor$}.
\end{equation}
We replace $f_\ell$ by 
\begin{equation}
f_\ell':=f_\ell-a_1G_1-\dots -a_{\lfloor\frac{\bl m}{24}\rfloor}G_{\lfloor\frac{\bl m}{24}\rfloor}\equiv f_\ell\pmod\ell.
\end{equation}
Then $f_\ell'$ vanishes to order $>\bl m/24$ at $\infty$, and vanishes to order $\geq 1$ at $0$.
Since $\eta^{\bl}(mz)$ has order $\bl m/24$ at $\infty$ and  order $\bl/24<1$ at $0$,
it follows that $f_\ell'(z)/\eta^{\bl}(mz)$ is a cusp form.  
The theorem follows in this case since 
by \cite[Cor. 3.5]{AAD} we have $\eta^{\bl}(mz)\in S_{\bl/2}\(m,\(\frac{\bullet}{m}\)\nu^{m\ell}\)$.
\end{proof}

The next lemma shows that Proposition~\ref{prop:ech} is not difficult  to prove  when $m$ is sufficiently large using modular forms of level one. 
For the remaining values of $m$, we rely on a variety of computations. 
\begin{lemma}\label{lem:finitem}
Proposition~\ref{prop:ech} is true if $\bl=1$ or if $m>523$.
\end{lemma}
\begin{proof}  
Since 
\begin{equation*}
    \frac{\ell+\bl-2}2\equiv \ell-1\not\equiv 2\pmod{12}
\end{equation*}
we have
\begin{equation}
\dim S_\frac{\ell+\bl-2}2(1)=\left\lfloor\frac{\ell+\bl-2}{24}\right\rfloor.
\end{equation}
Defining $\ell^*\in \{1, \dots, 11\}$ by  $\ell^* \equiv \ell\pmod{12}$ we find that 
\begin{equation}
\dim S_\frac{\ell+\bl-2}2(1)=\frac{\ell+\bl-2\ell^*}{24}.
\end{equation}
If $\dim S_k(1)=d$ then $S_k(1)$ has a basis $\{G_{j}\}_{j=1}^d$  with 
$G_j=q^j+O(q^{j+1})\in \Z[\![q]\!]$.
So the lemma will follow assuming that we have the inequality
\begin{equation}
\label{eq:weightineq}
\frac{\ell+\bl-2\ell^*}{24}\geq \frac{\bl m-1}{24}.
\end{equation}
By assumption we have $\ell\geq (m^2+24)/\bl$, so \eqref{eq:weightineq} will be true if
\begin{equation}
\label{eq:weightineq1}
m^2-\bl^2m +\bl^2+24-\bl(2\ell^*-1)\geq 0.
\end{equation}
It is easily checked that \eqref{eq:weightineq1} holds for all $m$ if $\bl=1$ and that it holds for $m>523$ for any value of $\bl$.
This proves the lemma.
\end{proof}

We next consider  small values of  $m$.
\begin{lemma}  Proposition~\ref{prop:ech} holds for $5\leq m\leq 31$.
\end{lemma}
\begin{proof}
Assume that $5\leq m\leq 31$.     For  $5\leq \ell\leq 23$ we compute explicit bases in  Sage \cite{sagemath} to see that the proposition holds for $S_{(\ell+\bl-2)/2}(m)=S_{\ell-1}(m)$.
    For   $\ell>23$ we may assume by  Lemma~\ref{lem:finitem} 
that $5\leq \bl\leq 23$.  We have just shown that 
  $S_{\bl-1}(m)$ has a basis of the form \eqref{eq:echbasis};   multiplying each element of this basis by $E_4^{(\ell-\bl)/8}$   (where $E_4$ is the Eisenstein series of weight $4$ on $\SL_2(\Z)$)  produces the desired basis for 
  $S_{(\ell+\bl-2)/2}(m)$.
\end{proof}

\begin{proof}[Proof of Proposition~\ref{prop:ech}]

After the two lemmas it remains to consider  $m$ with $37\leq m\leq 523$.
Let $g$ be the genus of $X_0(m)$.
We first assume that  $37\leq m\leq 269$; for these $m$ we check using Sage that $S_2(m)$
has a basis of forms with integer coefficients of the form
\begin{equation}\label{eq:s2basis}
H_1=a_1q+O(q^2), \ \ H_2=a_2q^2+O(q^3), \dots, H_g=a_gq^g+O(q^{g+1}),
\end{equation}
and with the property that 
\begin{equation}
    \ell\bl>m^2\implies \ell\nmid a_1 a_2\cdots a_g.
\end{equation}
Let $E=(m-1)+24q+\dots\in M_2(m)$ be the Eisenstein series with integer coefficients (note that  $\ell\nmid(m-1)$ by \eqref{eq:lmcond}).
Write $2k=(\ell+\bl-2)/2$.  By taking $k$-fold products of appropriate elements of the set 
\begin{equation}
    \{E, H_1, H_2, \dots, H_g\}
\end{equation}
we can produce forms in $S_{(\ell+\bl-2)/2}\(m,\Z\)$  of the form
\begin{equation}\label{eq:goodbasis}
G_1=b_1q+O(q^2), \ \ G_2=b_2q^2+O(q^3), \ \cdots, G_{gk}=b_{gk}q^{gk}+O(q^{gk+1}),
\end{equation}
and with the property that 
\begin{equation}\label{eq:s2int}
    \ell\bl>m^2\implies \ell\nmid b_1 b_2\cdots b_{gk}.
\end{equation}
The proposition will  follow from \eqref{eq:goodbasis} and \eqref{eq:s2int} for this range of $m$ assuming that we have the inequality
\begin{equation}
    \label{eq:genusineq}
    gk>\frac{\bl m}{24}.
\end{equation}
We have $g=\dim S_2(m)=\dim S_{m+1}(1)\geq (m-13)/12$
and 
\begin{equation}
    k=\frac{\ell+\bl-2}4\geq \frac{2\sqrt{\ell\bl}-2}4\geq\frac{\sqrt{m^2+24}-1}2.
\end{equation}
Since $\bl\leq 23$ we see that \eqref{eq:genusineq} is implied by 
\begin{equation}
    (m-13)(\sqrt{m^2+24}-1)>23m.
\end{equation}
The last inequality holds if $m>36.5$, which establishes the proposition for
$37\leq m\leq 269$.

Finally we assume that $271\leq m\leq 523$.  
For these $m$ we check using Sage that there is a basis of the form
\eqref{eq:s2basis} 
with the property that 
\begin{equation}\label{eq:s2int1}
    \ell\bl>m^2\implies \ell\nmid a_1 a_2\cdots a_{g-2}.
\end{equation}
We may now take $k$-fold products of elements of the set
\begin{equation}
    \{E, H_1, H_2, \dots, H_{g-2}\},
\end{equation}
and the proposition will follow if we have the inequality
\begin{equation}
    \label{eq:genusineq1}
    (g-2)k>\frac{\bl m}{24}.
\end{equation}
Arguing as above we see that this holds if $m>60.4$.  This proves the proposition.
\end{proof}

It remains to prove  Theorem~\ref{thm:Flconstruct} under the assumption \eqref{eq:smalllmcond}; for this we use explicit computations for each $m$ to address those values of $\ell$ which are not covered by \eqref{eq:lmcond}.

When $m\in \{5, 7\}$ there is nothing to prove in view of  \eqref{eq:cpmram} and  \eqref{eq:lmcond}.
For $m=11$ it suffices to prove Theorem~\ref{thm:Flconstruct} for $\ell\in\{73, 97\}$, and
for $m=13$ it suffices to prove the result for $\ell\in \{7, 11, 29, 73, 97\}$.  
For each of these cases we  verify by direct computation that 
there is a cusp form $f_\ell\in S_{(\ell+\bl)/2}(m)$ satisfying the conclusion of Lemma~\ref{lem:flconstruct} 
and such that $f_\ell$ vanishes to order $>\bl m/24$ at $
\infty$.  Dividing $f_\ell$ by $\eta^{\bl}(mz)$ then shows that Theorem~\ref{thm:Flconstruct} holds.

First suppose that $m=11$.  
When $\ell=73$ we find using Sage a basis $h_i=q^i+\dots$, $1\leq i\leq 34$ for $S_{36}(11)$ of modular forms with integer coefficients and verify (recalling \eqref{eq:cpmram}) that there are integers $\alpha_i$ with 
\begin{equation}
\eta(11z)\sum c\phi_{11}\pfrac{73n+11}{24}q^\frac n{24}\equiv \sum_{i=1}^{34}\alpha_i h_i\pmod{73}.
\end{equation}

When $\ell=97$ we find a basis $h_i$, $1\leq i\leq 46$ for $S_{48}(11)$ of modular forms with integer coefficients and verify that there are integers $\beta_i$ with 
\begin{equation}
\eta(11z)\sum c\phi_{11}\pfrac{97n+11}{24}q^\frac n{24}\equiv \sum_{i=1}^{46}\beta_i h_i\pmod{97}.
\end{equation}

Finally suppose that $m=13$.  In this case we have the  formula  \cite[(1.16)]{Chan-Wang-Yang} 
\begin{equation}\label{eq:cp13}
c\phi_{13}(n)=p(n/13)+13 p(13n-7)+26 a(n),\qquad\text{where}\qquad\sum_{n=0}^\infty a(n)q^n:= q\prod_{n=1}^\infty\frac{1-q^{13n}}{(1-q^n)^2}.
\end{equation}
We discuss only the cases $\ell=29$ and $97$ since the others are similar.
When $\ell=29$ we construct a basis for the $17$-dimensional $S_{16}(13)$ as follows:  There is a cusp form $h\in S_4(13)$ with integer coefficients which
is  identified by its expansion $h=q^2 - 3q^3 + q^4+\cdots$.
For $1\leq i\leq 16$ define $h_i\in S_{16}(13)$ by 
\begin{equation}\label{eq:weight16basis}
h_i:=\eta(13z)^{2 i - 6} \eta(z)^{30 - 2 i} h=q^i+\cdots.
\end{equation}
We find that there are integers $\alpha_i$ with 
\begin{equation}
\eta^5(13z)\sum c\phi_{13}\pfrac{29n+13}{24}q^\frac n{24}\equiv \sum_{i=3}^{16}\alpha_i h_i\pmod{29},
\end{equation}
and the desired result follows.

When $\ell=97$ we construct a basis for the $55$-dimensional $S_{48}(13)$ as follows:  
For $1\leq i\leq 55$ define $h_i\in S_{48}(13)$ by 
\begin{equation}\label{eq:weight48basis}
h_i:=\eta(13z)^{2 i - 8} \eta(z)^{104 - 2 i}=q^i+\cdots.
\end{equation}
We find that there are integers $\beta_i$ with 
\begin{equation}
\eta(13z)\sum c\phi_{13}\pfrac{97n+13}{24}q^\frac n{24}\equiv \sum_{i=1}^{55}\beta_i h_i\pmod{97}
\end{equation}
and the desired result follows.
This finishes the proof of Theorem~\ref{thm:Flconstruct}.

\begin{remark}
We  note that  there is no form in  in the 3-dimensional space $S_{4}(13)$ which is congruent modulo $5$ 
to 
\begin{equation}
\eta^5(13z)\sum c\phi_{13}\pfrac{5n+13}{24}q^\frac n{24}
\equiv 2q^3+2q^5+\cdots\pmod 5.
\end{equation}
So Theorem~\ref{thm:Flconstruct} does not hold when $m=13$ and $\ell=5$.
\begin{equation}
\end{equation}
\end{remark}

\section{Proof of Theorem~\ref{thm:frob1} and Theorem~\ref{thm:frob2}}
These theorems rely on the modular forms constructed in Theorem~\ref{thm:Flconstruct} together with the results 
of the recent paper \cite{AAD}, which give precise arithmetic information about the 
properties of the Shimura lifts of these forms.  To apply these results we need to prove that all Galois representations attached 
to newforms in the relevant integral weight spaces have large image modulo $\ell$ (see the definition below). 
Most of the work in this section involves proving that this occurs (this is similar to the proof  of \cite[~Proposition $3.3$]{Ahlgren-Allen-Tang}). 
We briefly describe the necessary notation; for full details see \cite[\S 3.1, \S 7.1]{AAD}.

 We assume  that $N$ is a positive square-free integer with $(N, 6)=1$, that   $\lambda\geq 2$ is an integer and that $\ell\geq 5$ is a prime with $\ell\nmid N$. We also assume that
  $\psi$ is a real Dirichlet character modulo $N$ and that $r$ is an integer with $(r, 6)=1$. 

Let $\bar{\Q}$ be the algebraic closure of $\Q$ in $\C$. For each prime $p$, fix an algebraic closure $\bar{\Q}_p$ of $\Q_p$ and an embedding $\iota_p:\bar{\Q}\rightarrow\bar{\Q}_p$.
For finite extensions $K/\Q$, let $G_K:=\Gal(\bar{\Q}/K)$.
The embedding $\iota_p$ allows us to view $G_p:=\Gal(\bar{\Q}_p/\Q_p)$ as a subgroup of $G_\Q$ and determines a prime above $\ell$ in each finite extension $K/\Q$. If $I_p$ is the inertia subgroup of $G_p$, then we denote the arithmetic Frobenius in $G_p/I_p$ by $\text{Frob}_p$. Let $\omega:G_\Q \rightarrow \F_\ell^\times$ denote the mod $\ell$ cyclotomic character.

If $\ep_2, \ep_3\in \{\pm 1\}$
let $S^{\new 2,3}_{2\lambda}(6N,\ep_2,\ep_3)$ be the 
space of  cusp forms of weight $2\lambda$ on $\Gamma_0(N)$ which are new at $2$ and $3$ and which have Atkin-Lehner eigenvalues $\ep_2$ and $\ep_3$ at $2$ and $3$ respectively.
Let  $f=\sum a(n)q^n\in S^{\new 2,3}_{2\lambda}(6N,\ep_2,\ep_3)$ be a newform (which necessarily has level $6N'$ for some $N'\mid N$)
and denote by 
$\bar{\rho}_{f}:G_{\Q} \rightarrow \GL_2(\bar{\mathbb{F}}_\ell )$
its associated residual Galois representation. We have $\text{tr} \ \bar{\rho}_f(\text{Frob}_p) \equiv a(p) \pmod{\lambda}$, where $\lambda$ is the prime above $\ell$ in the coefficient field of $f$.

Define
\begin{equation}\label{eq:epdef}
\ep_{p, r, \psi}:=-\psi(p)\pfrac{4p}{r}\ \ \ \ p\in \{2, 3\}.
\end{equation}
We recall a technical definition from \cite[\S 1]{AAD}.
\begin{definition}
With assumptions as above, we  say that the pair $(2\lambda,\ell)$ is {\it suitable} for the triple $(N,\psi,r)$ if for every newform $f \in S^{\new 2,3}_{2\lambda}(6N,\ep_{2,r,\psi},\ep_{3,r,\psi})$, the image of $\bar{\rho}_f$ contains a conjugate of 
$\SL_2(\mathbb{F}_\ell)$.
\end{definition}

The following is a consequence of  \cite[Thm. 1.3]{AAD}.
\begin{theorem}\label{thm:cong1}
With assumptions as above, 
suppose that 
\begin{equation}
F(z) =\displaystyle \sum_{n \equiv r \spmod{24}} a(n)q^{\frac{n}{24}} \in S_{\lambda+\frac{1}2}(N, \psi\nu^{r})
\end{equation}
has $\ell$-integral rational coefficients 
and that $(2\lambda,\ell)$ is suitable for $(N,\psi,r)$.
Then there is a positive density set $S$ of primes such that if $p \in S$, then $p \equiv 1 \pmod{\ell} $ and
\begin{equation}
a(p^2n) \equiv 0 \pmod{\ell}  \ \ \ \ \text{ if } \ \ \ \(\frac{n}{p}\)=
\(\frac{-1}{p}\)^{\frac{r-1}{2}}\psi(p).
\end{equation}
\end{theorem}

\begin{remark}
See the remark after 
Theorem~\ref{thm:frob2} for  a discussion
of what is meant by positive density.
\end{remark}

\begin{proof}[Proof of Theorem~\ref{thm:frob1}]
Let $\ell, m \geq 5$ be distinct  primes satisfying 
 \eqref{eq:lmcond} or \eqref{eq:smalllmcond}. 
 When $\ell, m\leq 11$  the result is trivial in view of \eqref{eq:cpramcong}. Thus, we will always assume that one of $\ell$ or $m$ is at least 13.
 By Theorem~\ref{thm:Flconstruct}
there exists $F_\ell\in S_{(\ell-2)/2}(m,\(\frac{\bullet}{m}\)\nu^{-m\ell})$ such that
\begin{equation}\label{eq:flrecall}
F_\ell(z)\equiv \sum \cpm\pfrac{\ell n+m}{24}q^\frac n{24}\pmod\ell.
\end{equation}
Therefore, Theorem~\ref{thm:frob1} for the pair ($\ell, m$) will follow from Theorem~\ref{thm:cong1} if we can verify that
\begin{equation}\label{eq:suitcheck}
    (\ell-3, \ell) \ \ \ \text{is suitable for}\ \ \ \(m, \(\tfrac{\bullet}{m}\), -m\ell\).
    \end{equation}
The next lemma  reduces us to a finite computation.
\begin{lemma}\label{lem:suit}
    Let $\ell, m \geq 5$ be distinct primes satisfying 
 \eqref{eq:lmcond} or \eqref{eq:smalllmcond}.  Then \eqref{eq:suitcheck} holds  with only these possible exceptions $(\ell, m)$:
 \begin{align*}
    \ell=7: \quad& m= 13,\\
    \ell=13: \quad&m=5, 7, 11,\\
    \ell=19: \quad&m=5, 7, 11, 13, 17.
\end{align*}
\end{lemma}

\begin{proof}
    By \cite[Prop 7.2]{AAD}  (with $k=\ell-3$) we see that \eqref{eq:suitcheck} holds if all of the following conditions are satisfied:

   \begin{gather}
    2^{\ell-4} \not \equiv 2^{\pm 1}\pmod\ell\label{eq:irreducible}\\
\ell-3 \neq \frac{\ell+1}2, \frac{\ell+3}2
\label{eq:ruleoutdihedral}\\
\frac{\ell+1}{(\ell+1,\ell-4)}, \frac{\ell-1}{(\ell-1,\ell-4)} \geq 6.
\label{eq:ruleoutexceptional}
   \end{gather}
It is straightforward to check that the first condition holds when $\ell\geq 7$, that the second holds when $\ell\geq 11$, and that the third holds when $\ell=11, 17$ and when $\ell\geq 23$.  This proves the lemma.
\end{proof}

To prove the theorem it remains to   show that \eqref{eq:suitcheck} holds for each pair $(\ell, m)$ in Lemma~\ref{lem:suit}.
Using \eqref{eq:epdef} we see that the spaces of interest   are 
$S^{\new 2,3}_{\ell-3}
\(6m,-\pfrac{8}{\ell},  -\pfrac{12}{\ell}\)$.
For each such pair let $f=\sum a(n)q^n\in S^{\operatorname{new} 2,3}_{\ell-3}
\(6m,-\pfrac{8}{\ell},  -\pfrac{12}{\ell}\)$ be a newform (necessarily of level $6$ or $6m$). It follows from \cite[Theorem $2.47(b)$]{DDT} that there are four possibilities for the image of $\bar{\rho}_f$:
\begin{enumerate}
\item
$\bar{\rho}_{f}$ is reducible.
\item
$\bar{\rho}_{f}$ is dihedral, i.e. $\bar{\rho}_{f}$ is irreducible but $\bar{\rho}_{f} \sl_{G_{K}}$ is reducible for some quadratic $K/\Q$.
\item
$\bar{\rho}_{f}$ is exceptional, i.e. the projective image of $\bar{\rho}_{f}$ is conjugate to one of $A_{4}$, $S_{4}$, or $A_{5}$.
\item
The image of $\bar{\rho}_{f}$ contains a conjugate of $\SL_2(\mathbb{F}_\ell )$.
\end{enumerate}
To rule out the first  possibility, we note that 
\eqref{eq:irreducible}  
holds for all of these pairs $(\ell, m)$. It follows from 
   \cite[Lemma $3.2$]{Ahlgren-Allen-Tang}   that $\bar{\rho}_f$ is not reducible.

By the same lemma, 
we need to rule out the second possibility only for pairs $(\ell, m)$ which do not satisfy \eqref{eq:ruleoutdihedral}; this reduces us to the single pair $(7, 13)$.
There are only two newforms  in the space 
$S^{\new 2,3}_{4}(78,-1,1)$. If $\bar{\rho}_{f}$ were dihedral for one of these $f=\sum a(n)q^n$, then we would have 
$\bar{\rho}_f \cong \bar{\rho}_f \otimes \omega^{3}$.  This would imply that $a(Q)\equiv 0\pmod 7$ for any prime $Q$ with $\pfrac Q7=-1$.  An examination of 
\cite{lmfdb} shows that for each of these forms $f$ (which have labels 78.4.a.d and 78.4.a.e)  we have $a(5)\not\equiv 0\pmod 7$.  Thus, the second possibility does not arise.

It remains to  rule out the third possibility.
Let $p \nmid 6\ell m$ be  prime and $u_f(p):=a(p)^2/p^{\ell-4}$.
If  $\bar{\rho}_f$ were exceptional, then we would have
 \begin{equation}\label{eq:ribet}
\bar{u_f(p)}=0,1,2,4 \ \ \text{ or } \ \bar{u_f(p)}^2-3\bar{u_f(p)}+1=0\ \ \ \text{in $\bar\F_\ell$,}
\end{equation}
depending on the order of the image of $\bar{\rho}_f(\operatorname{Frob}_p)$ in $\operatorname{PGL}_2(\bar{\mathbb{F}}_\ell)$ (see \cite[p. $264$]{Ribet1} and \cite[p.~189]{Ribet2}).
We compute in Magma to show  for each pair $(\ell, m)$
in Proposition~\ref{lem:suit}  
and for each newform
$f =\sum a(n)q^n\in S^{\operatorname{new} 2,3}_{\ell-3}
\(6m,-\pfrac{8}{\ell},  -\pfrac{12}{\ell}\)$ 
that there exists 
$p \in \{5,7,11,17,23\}$ such that \eqref{eq:ribet} does not hold for $u_f(p)$.  It follows that none of the $\bar\rho_f$ are exceptional.
\end{proof}

We now provide a detailed description  of our computations when $\ell=19$ and $m\in \{5,13\}$.
\begin{example}[The case $\ell=19$, $m\in \{5, 13\}$]

Here we need to consider   newforms
 $f\in S^{\operatorname{new}}_{16}(6m,1,1)$. 
There is one newform $f_0\in S^{\operatorname{new}}_{16}(6,1,1)$; it has integer coefficients 
and its reduction mod~$19$ is 
\begin{equation*}
q + 5q^2 + 17q^3 + 6q^4 + 17q^5+\cdots.
\end{equation*}
We find that $\bar{u_{f_0}(11)}=5$ and $\bar{u_{f_0}(11)}^2-3\bar{u_{f_0}(11)}+1=11$; it follows that $\bar{\rho}_{f_0}$ is not 
exceptional.

Suppose that $m=5$.  There are $8$ Galois orbits of newforms in $S^{\operatorname{new}}_{16}(30)$.
There are two  newforms $f_1, f_2\in S^{\operatorname{new}}_{16}(30,1,1)$, which can be  identified by 
\[\begin{aligned}
f_1&=q - 128q^2 - 2187q^3 + 16384q^4 + 78125q^5+\cdots,\\
f_2&=q - 128q^2 - 2187q^3 + 16384q^4 - 78125q^5 +\cdots.
\end{aligned}\]
The smallest witnesses  $p$ to the fact that the Galois representations are not exceptional are below.

\[
\renewcommand{\arraystretch}{1.5}
\begin{tabular}{|c|c|c|c|c}
\hline
    $f$ & $p$ &   $\bar{u_f(p)}$  & $\bar{u_f(p)}^2-3\bar{u_f(p)}+1$  \\ \hline \hline
    $f_1$ & $23$ & $7$ & $10$   \\ \hline
    $f_2$ & $11$ & $5$ & $11$   \\ \hline
      \end{tabular}
\]

In the case when $m=13$ there are $8$ Galois orbits of newforms in $S^{\operatorname{new}}_{16}\(78\)$.
Two of these orbits are contained in $S^{\operatorname{new}}_{16}\(78,1,1\)$, and
representatives of these orbits can be identified by 
\[\begin{aligned}
&q - 128q^2 - 2187q^3 + 16384q^4 + aq^5+\cdots,\\
& q- 128q^2 - 2187q^3 + 16384q^4 + bq^5 +\cdots, 
\end{aligned}\]
where 
\[\begin{gathered}
a^3 - 150844a^2 - 14396407620a +  2064376814439600=0,\\
b^4 - 169580b^3 - 64062939000b^2 +
    4905613268788000b + 874069880070911440000=0.
\end{gathered}\]
Since $19$ remains prime in $\Q(a)$, there  are three mod~$19$ reductions of forms in the first orbit.
These have coefficients  in $\F_{19^3}$ and can be identified as
\begin{equation*}
f_1=q + 5q^2 + 17q^3 + 6q^4 + x^{3508}q^5+\cdots  \qquad\text{where $\F_{19^3}^*=\langle x\rangle$},
\end{equation*}
and its two images under iterations of the  Frobenius $x\mapsto x^{19}$.

On the other hand, $19$ splits completely in $\Q(b)$, so the second orbit gives rise to four mod~${19}$ reductions
with coefficients in $\F_{19}$ which can be identified by 

\[\begin{aligned}
f_2&=q + 5q^2 + 17q^3 + 6q^4 +  16q^5+\cdots,\\
f_3&=q + 5q^2 + 17q^3 + 6q^4 +  13q^5+\cdots,\\
f_4&=q + 5q^2 + 17q^3 + 6q^4 +  10q^5+\cdots,\\
f_5&=q + 5q^2 + 17q^3 + 6q^4 +  4q^5+\cdots.\\
\end{aligned}\]
The smallest $p$ which witness the fact that the representations are not exceptional are below.

\[
\renewcommand{\arraystretch}{1.5}
    \begin{tabular}{|c|c|c|c|c}
\hline
    $f$ & $p$ &   $\bar{u_f(p)}$  & $\bar{u_f(p)}^2-3\bar{u_f(p)}+1$  \\ \hline \hline
    $f_1$ & $5$ & $x^{4730}$ & $x^{1158}$   \\ \hline
    $f_2$ & $7$ & $11$ & $13$   \\ \hline
     $f_3$ & $7$ & $7$ & $10$   \\ \hline
     $f_4$ & $5$ & $17$ & $11$   \\ \hline
     $f_5$ & $5$ & $5$ & $11$   \\ \hline     
      \end{tabular}
\]

\end{example}

\begin{proof}[Proof of Theorem $1.2$] 
  This result  follows directly  from an application of \cite[Theorem $(1.4)$]{AAD} to the form $F_\ell$ in \eqref{eq:flrecall}.
\end{proof}

\bibliographystyle{plain}
\bibliography{congruences-frobenius-partitions}
\end{document}